\documentclass [11pt]{amsart}

\usepackage{amsmath,amsthm, amsfonts, amssymb, setspace}

      \newtheorem{theorem}{Theorem}
      \newtheorem{corollary}{Corollary}
      \newtheorem{lemma}{Lemma}

      \def\R{{\mathbb R}}
      \def\N{{\mathbb N}}

      \def\cFdd'{\mathcal F_{dd'}}

      \def\ip{int\,P}
      \def\p{\psi}

\begin{document}

\title{An Extended Reich Fixed Point Theorem}

\author[Balasubramanian]{Sriram Balasubramanian}
\address{Department of Mathematics \& Statistics\\
  Indian Institute of Science Education and Research, Kolkata, India.}
\email{bsriram@iiserkol.ac.in}



\maketitle




\begin{abstract}
We obtain an extended Reich fixed point theorem for the setting of 
generalized cone rectangular metric spaces without assuming the 
normality of the underlying cone. Our work is a generalization of 
the main result in \cite{AAB} and \cite{JS}.
\end{abstract}


\thispagestyle{empty}

\section{Introduction}
A well-known fixed point theorem for metric spaces is the Banach
contraction mapping theorem, which states that if $(X,d)$ is a complete
metric space and the map $T:X \rightarrow X$ is a contraction, i.e.
$d(Tx,Ty) \le \lambda d(x,y)$ for some $\lambda \in [0,1)$ and
all $x,y \in X$, then there exists a unique fixed point for the map
$T$. In \cite{K}, Kannan considered a self-map $T$ on a complete metric
space $(X,d)$ that satisfies
\[
d(Tx,Ty) \le \alpha\{d(x,Tx) + d(y,Ty)\}
\]
for all $x,y \in X$. He proved that there exists a unique fixed
point for the map $T$, if $\alpha \in [0, 1/2).$ Several successful
attempts have been made to improve the Banach and Kannan fixed point
theorems, mainly along three different directions: (a) by finding
better contractivity conditions on the map $T$,
 (b) by replacing the underlying metric space with a more general space, 
for example - a partial metric space, a generalized metric space, an 
ordered metric space etc., and (c) by relaxing the completeness assumption. 
A small sample of such results can be found in \cite{AR} - \cite{C4}, \cite{R}. One 
such attempt is due to Reich, who in \cite{Re} proved that if $T$ is a self-map 
on a complete metric space $(X,d)$ that satisfies 
\[
d(Tx,Ty) \le \alpha d(x,Tx) +\beta d(y,Ty) + \mu d(x,y),
\]
where $\alpha, \beta, \mu \ge 0$ and $\alpha + \beta + \mu < 1$, then 
$T$ has a unique fixed point. \\

In \cite{HZ}, Huang and Zhang introduced the notion of a (normal)
cone metric, which is more general than a metric, and proved the 
Banach contraction mapping theorem for that setting. This initiated a
series of articles generalizing the Banach, Kannan and other
fixed point theorems to (normal) cone metric spaces. Hamlbarani and 
Rezapur Sh., in \cite{Rh}, further generalized the results of \cite{HZ} 
by dropping the normality assumption of the cone. Along the lines of
 \cite{B4}, Azam et al. in \cite{AAB}, introduced the
notion of a cone rectangular metric, which is obtained by replacing
the triangle inequality in the definition of a cone metric by a rectangular
 inequality, i.e. an inequality that involves four points and proved the
 Banach fixed point theorem for cone rectangular metric spaces.
 Jleli and Samet in \cite{JS} proved the Kannan fixed
point theorem for such spaces. Further generalization to the setting of 
TVS-cone rectangular metric spaces was done by Abdeljawad et al. 
in \cite{AT}. In this article, we obtain an extended Reich fixed point theorem for the 
setting of generalized cone rectangular metric spaces. Being inspired by and using 
the techniques in \cite{Rh}, we also do away with the normality of the 
underlying cone. This work generalizes the results found in \cite{AAB} and \cite{JS}. 


\section{Preliminaries and The Main Result}
\label{sec:prelms}
Let E be a real Banach space. A non-empty closed subset P of E is said
to be a $cone$ if
\begin{itemize}
\item [(a)] $P + P \subset P$
\item [(b)] $\alpha P \subset P$ for all $\alpha \in [0, \infty)$
\item [(c)] $P \cap (-P) = \{0\}$.
\end{itemize}

The cone $P$ is said to be $solid$ if the interior of $P$, which
we will denote by $int\,P$, is non-empty.\\
\newline
\textbf{Examples of Solid Cones:}
\begin{itemize}
\item [(1)] Let $E = \R$ and $P = [0, \infty)$.
\item [(2)] Let $E = \R^2$ and $P = \{(x,y)\,:\,x,y \ge 0\}$.
\end{itemize}

A cone $P$ in a real Banach space E, induces the following partial 
order $\le$ on E. For $x,y \in E$, 
\[
x \le y \, \Leftrightarrow \, y - x \in P.
\]
In the case of a solid cone $P$, we will use the notation
$x \ll y$ to denote $y - x \in int\,P$.\\

A cone $P$ is said to be $normal$ if for all $x, y \in P$ such that 
$x \le y$, there exists a constant $\kappa \ge 1$ such that $\|x\| \le 
\kappa \|y\|$. The examples (1) and (2) above 
are normal cones with $\kappa = 1$. An example of a 
cone that is not normal is the following (See \cite{Rh}). Let $E$ be 
the real Banach space $C^{\prime}[0,1]$ with the norm defined as $\|f\| = 
\|f\|_{\infty } + \|f^{\prime}\|_{\infty}$ and $P$ be the 
cone $\{f \,:\, f \ge 0\}$.\\


Let X be a nonempty set, $E$ be a real Banach space and
$P \subset E$ be a solid cone. A map $d\,:\,X \times X \rightarrow E$
is said to be a {\it{generalized cone rectangular metric}}, if there
exists a least integer $s \ge 1$ such that for all $x,y \in X$ and for all distinct
elements $u,v \in X \setminus \{x,y\}$,
\begin{itemize}
\item [(a)] $d(x,y) \ge 0$, i.e. $d(x,y) \in P$.
\item [(b)] $d(x,y) = 0$ if and only if $y = x$.
\item [(c)] $d(x,y) \le d(x,u) + s d(u,v) + d(v,y)$.
\end{itemize}
The pair $(X,d)$ is called a {\it{generalized cone rectangular metric
space.}} The quantity $s$ is called a weight of $(X,d)$.\\\\
\textbf{Example:} Let $X = \{1,2,3,4\}$, $E$ and $P$ be as in Example (2)
above. Define $d\,:\,X \times X \rightarrow E$ by
\[
d(1,2) = d(2,1) = (4,6);
\]
\[
d(1,3) = d(1,4) = d(2,3) = d(2,4) = d(3,4) = (1,1);
\]
\[
d(k,k) = 0 \text{ and } d(i,j) = d(j,i); \,\,\, i, j, k \in X.
\]
$(X,d)$ is a generalized cone
rectangular metric space with $s \le 4$.\\
\newline
\textbf{Remark 2:} Letting $s = 1$ in the definition
of the generalized cone rectangular metric $d$ yields the
definition of a cone rectangular metric, which was introduced by
Azam et al. in \cite{AAB}. Thus the collection of all generalized cone
rectangular metric spaces includes all the cone rectangular metric
spaces. Moreover, the inclusion is strict. Observe that 
$d(1,2) \not\le d(1,3) + d(3,4) + d(4,2).$\\

The following is our main result, which concerns
 a {\it{weakly complete generalized cone rectangular metric space}}.
 i.e. a space where every Cauchy sequence $\{x_n\}$ converges to
some $x$ in the space, in a weak sense namely,
given a neighborhood $U$ of $x$ {\it{of a certain type}}, there
exists a natural number $N$ such that $x_n \in U$ for all $n \ge N$.
Please see Section \ref{sec:gcrms} for the actual definition of
{\it{weak convergence.}}

\begin{theorem}
\label{thm:mainresult}
Let $(X,d)$ be a weakly complete generalized cone rectangular metric space 
with weight $s \ge 1$, $T:X \rightarrow X$ be a map. If for every $x,y \in X$, 
\begin{equation}
\label{eq:condnond}
d(Tx,Ty) \le \mu d(x,y) + \alpha d(x,Tx) + \beta d(y,Ty) + \gamma d(y,Tx),
\end{equation}
where $\mu, \alpha, \beta, \gamma \ge 0$, $0 \le \mu + \alpha + \beta < 1$, $\mu < \frac1s$ and $\mu + \gamma \in [0,1)$,
then $T$ has a unique fixed point.
\end{theorem}
The proof of the above Theorem is given in Section \ref{sec:main}. Choosing 
$s = 1$, the main result of \cite{AAB} can be recovered by taking $\alpha = \beta = \gamma = 0$, and the main result of \cite{JS} can be recovered by taking $\alpha = \beta$ and $\mu = \gamma = 0$. Also, here we do not assume the solid cone $P$ to be normal, as was the case in \cite{JS}. The absence of the normality assumption on the cone is handled by methods given in \cite{Rh}.

\section{Some Properties of a Cone in a Real Banach Space}
Recall the definition of a cone in a real Banach space given in
Section \ref{sec:prelms}.
The following are some well-known properties of a cone $P$
in a real Banach space $E$. A few more of them can be found
in \cite{HZ}, \cite{JKR}. 
\begin{lemma}
\label{lem:coneelemprops}
Let $E$ be a real Banach space and $P$ be a cone in $E$.
\begin{itemize}
\item [(i)] $P + int\,P \subset int\,P$
\item [(ii)] $\alpha \, (int\,P) \subset int\,P$ for all
$\alpha \in [0, \infty)$.
\item [(iii)] If $u \in P$ and $u \le ku$ for some $k \in [0,1)$,
then $u = 0$.
\end{itemize}
%
%
\end{lemma}

\begin{lemma}
\label{lem:coneprops}
Let $E$ be a real Banach space, $P$ be a solid cone in $E$
and $u,v,w \in E$.
\begin{itemize}
\item [(i)] If $0 \ll v \le u$, then $0 \ll u$. i.e. $u \in \ip$.
\item [(ii)] If $u \le v$ and $v \ll w$, then $u \ll w$.
\item [(iii)] If $0 \le u \ll c$ for all $c \gg 0$, then $u = 0$.
\item [(iv)] If $(x_n)$ is a sequence in $E$ such that $x_n \ge 0$
and $x_n \rightarrow 0$, then given $c \gg 0$, there exists
$N \in \N$ such that $x_n \ll c$ for all $n \ge N$.
\end{itemize}
\end{lemma}
%
%
%
%

\section{More on generalized cone rectangular metric spaces}
\label{sec:gcrms}
Recall the definition of a generalized cone rectangular metric
space given in Section \ref{sec:prelms}. For $x \in X$ and
$c \gg 0$, define $B(x,c) = \{y : d(x,y) \ll c\} \subset X$.
The collection $\mathcal B = \{B(x,c) : x \in X, c \gg 0\}$
being a subbasis generates a topology on X, say $\Gamma$.
Note that $\Gamma$ consists of all unions of finite intersections
of elements of $\mathcal B$. In particular $\mathcal B \subset
\Gamma$. We will henceforth view $(X, \Gamma)$ as a
topological space. That whether the topological space $(X,\Gamma)$
is Hausdorff, remains to be seen. The following definitions are adapted 
from \cite{AAB}, \cite{HZ}, \cite{JS}.\\

A sequence $(x_n)$ in a generalized cone rectangular metric space X is
said to be {\it Cauchy}, if given $c \gg 0$, there exists $N \in \N$,
which is independent of $k$, such that $d(x_n,x_{n+k}) \ll c$ whenever
$n \ge N$.\\

A sequence $(x_n)$ in a generalized cone rectangular metric space X is
said to {\it converge weakly} to $x \in X$, if given $c \gg 0$, there exists
an $N \in \N$ such that $d(x_n,x) \ll c$ for all $n \ge N$.
i.e. $x_n \in B(x,c)$ for all $n \ge N$. We will denote $(x_n)$ converging
 weakly to $x$ by $x_n \twoheadrightarrow x$. \\

A generalized cone rectangular metric space $X$ is said to be {\it weakly
complete} if every Cauchy sequence in the space converges weakly. \\
\newline
\textbf{Example:} Let $(X,d)$ be the generalized cone rectangular
metric space given in the Example in Section \ref{sec:prelms}.
It can easily be verified that the only Cauchy sequences in $X$ are
the eventually constant sequences, which of course converge weakly
in $X$. Thus $(X,d)$ is a complete generalized cone rectangular metric
space.\\
\newline
\textbf{Remark 4:} For a general topological space $Y$, one says that
a sequence $(x_n)$ in Y converges to $x$ if and only if given
any open set $U$ containing $x$, there exists an $N \in\N$ such that
$x_n \in U$, for all $n \ge N$. Observe that for our purposes, we
only consider a weaker form of convergence. The type of convergence
defined above is weaker, because we demand the existence of an
$N \in \N$ not for all, but only for certain open sets containing
$x$, namely, sets of the form $B(x,c)$.\\

The following lemma is a minor variant of Lemma 1.10 from \cite{JS}.
\begin{lemma}
\label{lem:samelim}
Let $(x_n)$ be a Cauchy sequence of distinct points in a generalized
cone rectangular metric space $X$. Let $x, y \in X$ and $M \in \N$ be such that 
$x,y \not \in \{x_n:n \ge M\}$. If $(x_n)$ converges weakly to both 
$x$ and $y$, then $x = y$.
\end{lemma}
\begin{proof}
Given $c \gg 0$, by hypothesis, choose $N \ge M$ such that 
such that $d(x,x_n) \ll \frac{c}{3}, \, \,
d(x_m,x_{m+1}) \ll \frac{c}{3s}$ and $d(x_{\ell},y) \ll \frac{c}{3}$, for all $m,n, \ell \ge N$.
Using part (i) of Lemma \ref{lem:coneelemprops} and part (ii) of
Lemma \ref{lem:coneprops}, it follows that
\[
d(x,y) \le d(x,x_N) + sd(x_N,x_{N+1}) + d(x_{N+1},y)\ll  c.
\]
An application of part (iii) of Lemma \ref{lem:coneprops}, completes
the proof.
\end{proof}


\section{Proof of The Main Result}
\label{sec:main}
This section contains a proof of our main result stated in
Section \ref{sec:prelms}. A part of it is adapted from      
Theorem 2.1 in \cite{JS} and Theorem 1 of \cite{MSS}. 


\begin{proof}[Proof of Theorem \ref{thm:mainresult}]
First we prove uniqueness of the fixed point. Suppose that $x, y \in X$ are
fixed points of $T$. From inequality (\ref{eq:condnond}), it follows that
\begin{align*}
d(x,y) &\le \mu d(x,y) + \alpha d(x,Tx) + \beta d(y,Ty) + \gamma d(y, Tx) \\
&= (\mu + \gamma) d(y,x).
\end{align*}
By part (iii) of Lemma \ref{lem:coneelemprops}, it follows that
$d(x,y) = 0$. Thus $x = y$.

Let $\delta = \frac{\mu + \alpha} {1-\beta}$. Note that
$\delta \in [0,1)$. Fix $x \in X$. For $n \in \N$, 
using inequality (\ref{eq:condnond}), we get
\[
d(T^nx, T^{n+1}x) \le (\mu + \alpha) d(T^{n-1}x, T^nx) +
\beta d(T^nx,T^{n+1}x).
\]
Thus for each $n \in \N$,
\[
d(T^nx, T^{n+1}x) \le \delta d(T^{n-1}x, T^{n}x).
\]
Iterating we get,

\begin{equation}
\label{eq:condnont}
d(T^nx, T^{n+1}x)\le \delta^n d(x,Tx).
\end{equation}

Without loss of generality, one can assume that the sequence $(T^nx)$ consists
of distinct elements. This is because if $T^mx = T^nx$ for some $m > n$, then
\begin{align*}
d(T^nx, T^{n+1}x) & = d(T^{m-n}(T^nx), T^{m-n+1}(T^nx))\\
                  & \le \delta^{m-n} d(T^nx, T^{n+1}x).
\end{align*}
From part (iii) of Lemma \ref{lem:coneelemprops}, it follows that
$d(T^nx, T^{n+1}x) = 0$, i.e. $T^nx$ is a fixed point of $T$.
Henceforth, we will assume $T^nx \neq T^mx$ for any $n \neq m$.

Next we prove that the sequence $(T^nx)$ is a Cauchy sequence
in $X$. For $k \in \N$, consider $d(T^nx, T^{n+k}x)$.

Suppose that $k$ is odd. We have,
\begin{align}
\notag
d(T^nx, T^{n+k}x) &\le \{d(T^nx,T^{n+1}x) + s \, d(T^{n+1}x,T^{n+2}x)
+ d(T^{n+2}x, T^{n+3}x) \\
\notag
                  & + s \,d(T^{n+3}x,T^{n+4}x) + \cdots + d(T^{n+k-1}x, T^{n+k}x)\}\\
\notag
& \le s\, \{d(T^nx,T^{n+1}x) + d(T^{n+1}x,T^{n+2}x) + d(T^{n+2}x, T^{n+3}x) \\
\notag
&+ d(T^{n+3}x,T^{n+4}x) + \cdots + d(T^{n+k-1}x,
T^{n+k}x)\} \\
& \le s(\delta^n + \delta^{n+1} + \cdots + \delta^{n+k-1})\, d(x,Tx)\\
\label{eq:kisodd}
& \le \frac {s \delta^{n}}{1 - \delta} \, d(x,Tx).
\end{align}

Suppose that $k = 2$. Using inequality (\ref{eq:condnond}), the
facts $0 \le \alpha, \beta, \gamma \le 1$ and $0 \le \mu < \frac1s$, we get
\begin{align}
\notag
d(T^nx,T^{n+2}x) & \le \{\mu d(T^{n-1}x, T^{n+1}x) + \alpha d(T^{n-1}x,T^nx)\\
\notag                 
&     + \beta d(T^{n+1}x,T^{n+2}x) + \gamma d(T^{n+1}x,T^nx)\}\\
                 \notag
&\le \{\mu[d(T^{n-1}x, T^nx) + s d(T^nx, T^{n+2}x) \\
\notag
&  + d(T^{n+2}x, T^{n+1}x)] + \alpha d(T^{n-1}x,T^nx) \\
\notag 
& + \beta d(T^{n+1}x,T^{n+2}x) + \gamma d(T^{n+1}x,T^nx)\}\\
\notag
& \le \{(\mu + \alpha) d(T^{n-1}x,T^nx) + (\mu + \beta)  d(T^{n+2}x, T^{n+1}x) \\
\notag
& + \gamma d(T^{n+1}x,T^nx) + \mu s d(T^nx, T^{n+2}x)\}\\
\notag
& \le \{(\mu + \alpha) \delta^{n-1} + (\mu + \beta) \delta^{n+1} + \gamma \delta^n\} d(x,Tx) \\
\notag
&+ \mu s d(T^nx, T^{n+2}x).\\
\notag
& \le  5 \delta^{n-1} d(x,Tx) + \mu s d(T^nx, T^{n+2}x).
\end{align}
Thus,
\begin{align}
\label{eq:kis2}
d(T^nx,T^{n+2}x) \le \frac{5 \delta^{n-1}}{1-\mu s} d(x,Tx).
\end{align}

Suppose that $k > 2 $ is even. We have 
\begin{align}
\notag
d(T^nx, T^{n+k}x) &\le \{d(T^nx,T^{n+2}x) + s \, d(T^{n+2}x,T^{n+3}x)
+ d(T^{n+3}x, T^{n+4}x) \\
\notag
                  & + s\, d(T^{n+4}x,T^{n+5}x) + \cdots + d(T^{n+k-1}x, T^{n+k}x)\}\\
\notag
& \le d(T^nx,T^{n+2}x) + s \{ d(T^{n+2}x,T^{n+3}x) + d(T^{n+3}x, T^{n+4}x) \\
\notag
&+ d(T^{n+4}x,T^{n+5}x) + \cdots + d(T^{n+k-1}x,
T^{n+k}x)\} \\
\notag
& \le \frac{5 \delta^{n-1}}{1-\mu s} d(x,Tx)+ s \{\delta^{n+2} + \cdots + \delta^{n+k-1}\}\, d(x,Tx)\\
\label{eq:kiseven}
& \le \left(\frac{5 \delta^{n-1} }{1-\mu s} + \frac {s \delta^{n}}{1 - \delta} \right)\, d(x,Tx).
\end{align}


It follows from
inequalities (\ref{eq:kisodd}), (\ref{eq:kis2}) and (\ref{eq:kiseven}) that
\begin{align}
\label{eq:firststage}
d(T^nx, T^{n+k}x) \le \left(\frac{5 \delta^{n-1} }{1-\mu s} + \frac {s \delta^{n}}{1 - \delta} \right)
d(x,Tx) \,\,\,\text{ for all } n, \,k \in \N.
\end{align}
Let $y_n = \left(\frac{5  \delta^{n-1} }{1-\mu s} + \frac {s \delta^{n}}{1 - \delta} \right) d(x,Tx) \in P$. For each $k \in \N$, it follows that 
\begin{align}
\label{eq:cauchy}
d(T^nx, T^{n+k}x) \le y_n.
\end{align}

Since the 
sequence $(y_n)$ in $P$ converges to $0$, for a given $c \gg 0$, using part (iv) of Lemma \ref{lem:coneprops}, one can choose a natural number $K$ such that 
$y_n \ll c$ for all $n \ge K$. Thus for each $k \in \N$ and $n \ge K$, we get 
\[
d(x_n,x_{n+k}) \le y_n \ll c.
\]
An application of part (ii) of Lemma \ref{lem:coneprops} implies that the sequence 
$(T^nx)$ is Cauchy. Since $X$ is weakly complete, there exists $u \in X$ such that $T^nx \twoheadrightarrow u$.
Note that the uniqueness of the limit $u$ of the sequence $(T^nx)$ is
guaranteed by Lemma \ref{lem:samelim}. Our next claim is that $u$ is
a fixed point of $T$.\\

Recall that $(T^nx)$ is a sequence in $X$ of distinct points. Choose $m 
\in \N$ such that $u, Tu \not \in \{T^kx\,:\, k \ge m\}$. Using inequality (\ref{eq:firststage}) together with part (iv) of Lemma \ref{lem:coneprops}, it follows that
for every $c \gg 0$, there exists $N \ge m$ such that 
\begin{align}
\label{eq:ineqs}
d(u,T^nx) \ll \frac{(1 - \beta)c}{6} \,\,\,\, \text{ and }\, \, \,\,
d(T^{\ell}x, T^{\ell+1}x) \ll \frac{(1 - \beta)c}{6s}.
\end{align}
 for all $n, \ell \ge N$. We have,
\begin{align}
\notag
d(u,Tu) & \le d(u,T^{N}x) + s d(T^{N}x,T^{N+1}x) + d(T^{N+1}x, Tu)\\
\notag
        & \le \{d(u,T^{N}x) + s d(T^{N}x,T^{N+1}x)  + \mu d(T^Nx,u) \\
\notag
        & + \alpha d(T^{N}x,T^{N+1}x)  + \beta d(u.Tu) + \gamma d(u, T^{N+1}x)\}\\
\notag
        & \le \{2d(u,T^Nx) + 2sd(T^{N}x,T^{N+1}x) + \beta d(u,Tu) \\
\label{eq:fixedpt}
       &+ 2 d(u,T^{N+1}x)\}.
\end{align}
Thus,
\begin{align*}
d(u,Tu) &\le \frac{2}{1 - \beta} \{d(u,T^Nx) + d(u,T^{N+1}x)\}  \\
& + \frac{2s}{1 - \beta} d(T^{N}x,T^{N+1}x)\\
         &\ll \, c.
\end{align*}
By part (iii) of Lemma \ref{lem:coneprops}, it follows that $\p(d(u,Tu)) = 0$. i.e. $u$ is a fixed point of $T$.
%
%
\end{proof}

\begin{corollary}
Let $E = \R$, $P = [0, \infty)$. If $(X,d)$ is a generalized cone rectangular metric space with weight $s \ge 1$ and $T:X \rightarrow X$ satsifies 
$d(Tx,Ty) \le \frac{1}{2} \{ d(Tx,x)  + d(Ty,y) \},$, then $T$ has a unique fixed point.
\end{corollary}
\begin{proof}
Taking $\mu = \gamma = 0$, $\alpha = \beta = \frac{1}{2}$ in Theorem \ref{thm:mainresult} yields the desired result.
\end{proof}

\begin{corollary}
Let $E = \R$, $P = [0, \infty)$. If $(X,d)$ is a generalized cone rectangular metric space with weight $s \ge 1$ and $T:X \rightarrow X$ satsifies 
$d(Tx,Ty) \le \mu d(x,y)$, where where $0 \le \mu < 1$, then $T$ has a unique fixed point.
\end{corollary}
\begin{proof}
Taking $\gamma = \alpha = \beta = 0$ in Theorem \ref{thm:mainresult} yields the desired result.
\end{proof}

\begin{corollary}
Let $E = \R$, $P = [0, \infty)$. If $(X,d)$ is a cone rectangular metric space and $T:X \rightarrow X$ satsifies 
$d(Tx,Ty) \le \alpha d(Tx,x) + \beta d(Ty,y) + \mu d(x,y),$ where $\alpha, \beta, \mu \ge 0$ with $\alpha + \beta + \mu < 1$, then $T$ has a unique fixed point.
\end{corollary}
\begin{proof}
Taking $s = 1$, $\gamma = 0$ in Theorem \ref{thm:mainresult} yields the desired result.
\end{proof}

\end{document}